\documentclass[12pt,reqno]{amsart}
\usepackage{fullpage}
\usepackage{times}
\usepackage{amsmath,amssymb,amsthm,url}
\usepackage[utf8]{inputenc}
\usepackage[english]{babel}
\usepackage{comment}
\usepackage{bbm}
\usepackage{enumerate}
\usepackage{bm}
\usepackage{graphicx}
\usepackage{mathrsfs}
\usepackage[colorlinks=true, pdfstartview=FitH, linkcolor=blue, citecolor=blue, urlcolor=blue]{hyperref}
\usepackage{tabstackengine}
\stackMath
\numberwithin{equation}{section}

\newtheorem{thm}{Theorem}[section]
\newtheorem{lem}[thm]{Lemma}
\newtheorem{prop}[thm]{Proposition}
\newtheorem{cor}[thm]{Corollary}
\newtheorem{conj}[thm]{Conjecture}

\theoremstyle{definition}

\newtheorem{question}[thm]{Question}

\newtheorem{rem}[thm]{Remark}

\newcommand{\N}{\mathbb{N}}

\newcommand{\F}{\mathbb{F}}

\title{Multiplicatively reducible subsets of shifted perfect $k$-th powers and bipartite Diophantine tuples}
\author{Chi Hoi Yip}
\address{School of Mathematics\\ Georgia Institute of Technology\\ GA 30332\\ United States}
\email{cyip30@gatech.edu}
\subjclass[2020]{11D45, 11D72, 11N36, 11B30}
\keywords{Diophantine tuples, multiplicative decomposition}

\begin{document}

\begin{abstract}
Recently, Hajdu and S\'{a}rk\"{o}zy studied the multiplicative decompositions of polynomial sequences. In particular, they showed that when $k \geq 3$, each infinite subset of $\{x^k+1: x \in \mathbb{N}\}$ is multiplicatively irreducible. In this paper, we attempt to make their result effective by building a connection between this problem and the bipartite generalization of the well-studied Diophantine tuples. More precisely, given an integer $k \geq 3$ and a nonzero integer $n$, we call a pair of subsets of positive integers $(A,B)$ \emph{a bipartite Diophantine tuple with property $BD_k(n)$} if $|A|,|B| \geq 2$ and $AB+n \subset \{x^k: x \in \mathbb{N}\}$. We show that $\min \{|A|, |B|\} \ll \log |n|$, extending a celebrated work of Bugeaud and Dujella (where they considered the case $n=1$). We also provide an upper bound on $|A||B|$ in terms of $n$ and $k$ under the assumption $\min \{|A|,|B|\}\geq 4$ and $k \geq 6$. Specializing our techniques to Diophantine tuples, we significantly improve several results by B\'{e}rczes-Dujella-Hajdu-Luca, Bhattacharjee-Dixit-Saikia, and Dixit-Kim-Murty. 
\end{abstract}

\maketitle

\section{Introduction}
Throughout, let $\N$ denote the set of positive integers and let $n$ be a nonzero integer. Let $p$ be a prime, $\F_p$ be the finite field with $p$ elements, and $\F_p^*=\F_p \setminus \{0\}$.

A subset $S$ of positive integers is said to be \emph{multiplicatively reducible} if it has a multiplicative
decomposition $S=AB$ with $|A|, |B| \geq 2$, where the product set $AB=\{ab: a \in A, b \in B\}$. There is extensive literature on the study of multiplicative decompositions in various settings; we refer to a nice survey by Elsholtz \cite{E09}.

Recently, Hajdu and S\'{a}rk\"{o}zy studied the multiplicative decompositions of polynomial sequences with integer coefficients in a series of papers \cite{HS18, HS18b, HS20} and they classified polynomial sequences that are multiplicatively reducible in \cite[Theorem 4.1]{HS18} and \cite[Theorem 3.1]{HS18b}. They also paid special attention to the following set of shifted $k$-th powers:
$S_k'=\{x^k+1: x \in \N\}.$ When $k=2$, they showed that $S_2'$ contains an infinite subset that is multiplicatively reducible \cite[Theorem 3.1]{HS18} using Pell's equations; see a related discussion in \cite[Theorem 3]{G01} by Gyarmati. On the other hand, they showed that when $k \geq 3$, each infinite subset of $S_k'$ is multiplicatively irreducible \cite[Theorem 2.1]{HS18b} using a result of Baker \cite[Lemma 2.1]{HS18b}. Their proof can be slightly modified to show the same statement for each infinite subset of any nontrivial shift of $k$-th powers, namely $\{x^k+n: x \in \N\} \cap \N$, where $k \geq 3$ and $n$ is a nonzero integer. One can also apply Siegel's theorem on integral points to the hyperelliptic (elliptic if $k=3$ or $4$) curve of the form $C: y^k=(a_1x+n)(a_2x+n)$ to give a simple proof of this fact. However, it seems these two approaches are both ineffective in the sense that they cannot be used to deduce a quantitative upper bound on the size of a subset of $S_k'$ (and more generally $\{x^k+n: x \in \N\} \cap \N$) which is multiplicatively reducible. The above discussions lead to the following question:
\begin{question}\label{q1}
Let $k \geq 3$ and let $n$ be a nonzero integer. Suppose that $S$ is a subset of $\{x^k+n: x \in \N\} \cap \N$ that is multiplicatively reducible. Can we give a quantitative upper bound on $|S|$?
\end{question}

To answer Question~\ref{q1}, we are naturally led to the following question:
\begin{question}\label{q2}
Let $k \geq 3$ and let $n$ be a nonzero integer. Suppose that $A,B \subset \N$ with $|A|, |B| \geq 2$ satisfy $AB +n\subset \{x^k: x \in \N\}$. Is $|A||B|$ bounded and can we give a quantitative upper bound on $|A||B|$?
\end{question}

Note that Question~\ref{q2} is essentially equivalent to Question~\ref{q1} but is slightly stronger, thus we shall focus on Question~\ref{q2} only in the following discussion. Our motivations to study these two questions also come from a very recent paper by Kim, Yip, and Yoo \cite{KYY}, where they studied the finite field analogue of these two questions and in particular made significant progress toward a conjecture of S\'{a}rk\"{o}zy on multiplicative decompositions of shifted multiplicative subgroups \cite{S14}.

It turns out that Question~\ref{q2} is closely related to the study of Diophantine tuples. There is a long history and a large amount of literature on the study of Diophantine tuples and their generalizations in various settings; we refer to the recent book of Dujella \cite{D24} for a comprehensive discussion. The following natural generalization of Diophantine tuples has been studied extensively (see for example \cite{BDHL11, BDS23, BD03, DKM22, D02, GM20, G01, KYY}): for each nonzero integer $n$ and $k \ge 2$, we call a set $C=\{c_{1}, c_{2},\ldots, c_{m}\}$ of distinct positive integers a \textit{Diophantine $m$-tuple with property $D_{k}(n)$} if $c_ic_j+n$ is a $k$-th power for each $1 \leq i<j \leq m$. Of special interest is to find a good upper bound on the following two quantities (we follow the notations used in \cite{BDS23, DKM22, KYY}):
$$
M_{k}(n)=\sup \{|C| \colon C\subset{\mathbb{N}} \text{ satisfies property }D_{k}(n)\},
$$
$$
M_k(n;L):=\sup\{|C \cap [|n|^L,\infty)|: C\subset{\mathbb{N}} \text{ satisfies property } D_k(n)\},
$$
where $L>0$ is a real number. Most notable is the Diophantine quintuple conjecture, that is, $M_2(1)=4$, recently confirmed by He, Togb\'e, and Ziegler \cite{HTZ19}.

Question~\ref{q2} is exactly the bipartite analogue of the question of estimating $M_k(n)$. Moreover, Question~\ref{q2} is more general in the sense that we have two sets instead of a single set. In particular, in the study of Diophantine tuples, gap principles play an important role (see for example \cite{BDHL11, BD03, DKM22, G01, KYY}). However, to use such a gap principle in the bipartite setting, one has to be careful since the ordering of the elements in the two sets is not known in advance; see Remark~\ref{rem:ordering}. Thus, studying bipartite Diophantine tuples is much more challenging. As a byproduct of our exploration of Question~\ref{q2}, we also obtain improved upper bounds on $M_k(n)$ and $M_k(n, L)$.

For convenience, we introduce the following analogous notions. For each $k \ge 3$ and each nonzero integer $n$, we call a pair of sets $(A, B)$ a \textit{bipartite Diophantine tuple with property $BD_{k}(n)$} if $A, B$ are two subsets of $\N$ with size at least $2$, such that $ab+n$ is a $k$-th power for each $a \in A$ and $b \in B$. We shall define the following three quantities:$$
BM_{k}(n)=\sup \{\min \{|A|,|B|\} \colon (A,B) \text{ satisfies property }BD_{k}(n)\},
$$
$$
PM_{k}(n):=\sup \{|A||B| \colon (A,B) \text{ satisfies property }BD_{k}(n)\},
$$
$$
PM_{k}(n, R):=\sup \{|A||B| \colon (A,B) \text{ satisfies property }BD_{k}(n) \text{ and } \min \{|A|, |B|\} \geq R\}.
$$
To see the motivation for studying the above three quantities, observe that, given a Diophantine tuple $C$ with property $M_k(n)$, we can always write $C=A \sqcup B$ (as a disjoint union) in a balanced way in the sense that $|A|$ and $|B|$ differ by at most $1$; since $(A, B)$ has property $BD_{k}(n)$, it follows that\begin{equation}\label{eq:compare}
M_k(n) \leq 2BM_{k}(n)+1.    
\end{equation}
Thus, whenever we have an upper bound on $BM_{k}(n)$, it immediately produces an upper bound on $M_k(n)$; moreover, it is intuitive that $M_k(n)$ is much smaller than $BM_k(n)$ since being a Diophantine tuple is much more restrictive. On the other hand, the quantity $PM_k(n)$ is exactly the one asked by Question~\ref{q2} and it also serves as an upper bound for Question~\ref{q1}. 

While the above three quantities appear to be new in general, the special case $n=1$ has been explored by Gyarmati \cite{G01}, and Bugeaud and Dujella \cite{BD03} (see also Bugeaud and Gyarmati \cite{BG04}), and we shall extend their results to all $n$ using a more general method. Nevertheless, even when $n=1$, Question~\ref{q2} is open according to Dujella~\footnote{private communication}. In other words, it is unknown if $PM_k(1)=PM_k(1,2)$ is finite; see also Remark~\ref{rem:ordering}. Thus, we are led to study the weaker quantity $PM_k(n, R)$ in the hope that if $R$ is a bit larger than $2$, then we can show that unconditionally $PM_k(n, R)$ is finite. In particular, among other new results, we show that this is indeed the case when $R=4$ and $k \geq 6$ in Theorem~\ref{thm:PM_k}.

\textbf{Notation.} We follow the Vinogradov notation $\ll$. We write $X \ll Y$ if there is an absolute constant $C>0$ so that $|X| \leq CY$.

\textbf{Structure of the paper.} 
In Section~\ref{sec:mainresults}, we state our new results and compare them with previous results. In Section~\ref{sec:prelim}, we provide additional background and prove some preliminary results. In Section~\ref{sec:proof}, we present the proof of our new results. In particular, we prove Theorem~\ref{thm:explicit}, Theorem~\ref{thm:M_k(n,3)}, and Theorem~\ref{thm:M_k(n)explicit} in Section~\ref{subsec:applications}. We then present the proof of Theorem~\ref{thm:M_k(n)} in Section~\ref{subsec:small}. Finally, we prove Theorem~\ref{thm:BM_k} and Theorem~\ref{thm:PM_k} in Section~\ref{subsec:mainproof}.

\section{New results}\label{sec:mainresults}
First, we discuss our contributions to $BM_k(n)$. In the case $n=1$, Bugeaud and Dujella \cite[Corollary 3]{BD03} showed that $BM_3(1) \leq 8$, $BM_4(1) \leq 4$, $BM_k(1) \leq 3$ for $k \geq 5$. The following two theorems generalize their results. Our first result provides an explicit upper bound on $BM_k(n)$. 
\begin{thm}\label{thm:explicit}
Let $k \geq 3$ and let $n$ be a nonzero integer. If $|n|=1$, then $BM_k(n)\leq r_k+1$, where
\begin{equation}\label{r_k}
r_3=9, \quad  r_4=6, \quad r_5=5, \quad \text{ and } \quad r_k=4 \quad \text{ for } k \geq 6.
\end{equation}
If $|n|\geq 2$, then $$BM_k(n) \leq \max\bigg\{\frac{\log \log |n|+3.3}{\log (k-1)}+8, (\max\{4|n|^2,|n|^{2(k+1)/(k-2)}\}+n)^{1/k}+20\bigg\}.$$
\end{thm}

When $k$ is fixed and $|n|$ is large, the above explicit upper bound can be substantially improved with the help of sieve methods.

\begin{thm}\label{thm:BM_k}
Let $k \geq 3$. As $|n| \to \infty$, we have
\begin{enumerate}
    \item [(1)] $BM_{k}(n) \leq (\frac{4\phi(k)}{k-2}+o(1)) \log |n|$;
    \item [(2)] Assuming the Paley graph conjecture, we have $BM_{k}(n)=(\log |n|)^{o(1)}$.
\end{enumerate}
\end{thm}

When $k$ is a composite, the unconditional upper bound on $BM_k(n)$ can be further improved; see Remark~\ref{composite}. 
The connection between the Paley graph conjecture (see Conjecture~\ref{Paley-graph-conjecture}) and Diophantine tuples was discovered by G\"{u}lo\u{g}lu and Murty \cite{GM20}. 

Next, we show that if $R \gg 1$, then $PM_k(n, R)<\infty$ and we give a quantitative upper bound. Moreover, if $R \gg \log \log |n|$, we show an improved upper bound on $PM_k(n, R)$.

\begin{thm}\label{thm:PM_k}
Let $\epsilon>0$. Let $k \geq 3$ and let $n$ be a nonzero integer. Then 
$$
PM_k(n, r_k)\ll |n|^{\frac{t_k}{k}+\epsilon}, \quad PM_k\bigg(n, \frac{\log \log |n|+3.3}{\log (k-1)}+8\bigg) \ll |n|^{\frac{1}{k-2}+\epsilon},
$$
where $r_k$ is defined in equation~\eqref{r_k}, and
$$t_3=\frac{15399}{938}, \quad t_4=\frac{34}{3}, \quad t_5=\frac{97}{23}, \quad t_6=\frac{29}{4}, \text{ and } \quad t_k=\frac{k^2+k-4}{k^2-6k+6} \quad \text{for} \; \;  k \geq 7.
$$
\end{thm}

We also specialize our techniques for Diophantine tuples and unexpectedly, we manage to improve several best-known bounds on $M_k(n)$ and $M_k(n, L)$ in the literature. 

Dujella \cite{D02} showed that $M_2(n,3) \leq 21$ for all nonzero integers $n$. Dixit, Kim, and Murty \cite[Theorem 1.3(a)]{DKM22} showed that 
$M_k(n,3) \ll_{k} 1$ if $k$ is fixed and $n \to \infty$. Using a similar idea, Bhattacharjee, Dixit, and Saikia \cite{BDS23} proved the same result when $n$ is negative. Moreover, they proved the explicit upper bound that 
\begin{equation}\label{eq:explciit}
M_k(n,3) \leq 2^{28} \log (2k) \log (2\log (2k))+14
\end{equation}
in \cite[Theorem 1.3 (a)]{BDS23} for $k \geq 3$ and $n \neq 0$ using an explicit version of Roth's theorem due to Evertse \cite{E10}. However, in their proof, they implicitly assumed that $|n|$ is sufficiently large compared to $k$ \cite[Lemma 3.2]{BDS23}, and thus, in fact, they only showed that inequality~\eqref{eq:explciit} holds when $|n|$ is sufficiently large compared to $k$ \footnote{private communication with Anup Dixit} (according to a related discussion in \cite[Section 2.3]{KYY}, it is sufficient to assume that $n>(2ke)^{k^2}$). We remove the dependence on $k$ and $n$ from the above upper bound on $M_k(n,3)$ and show that $M_k(n,\frac{k}{k-2})$ is already uniformly bounded. 

\begin{thm}\label{thm:M_k(n,3)}
Let $k \geq 3$ and let $n$ be a nonzero integer. If $|n|=1$, then $M_k(n) \leq r_k$, where $r_k$ is defined in equation~\eqref{r_k}. 
If $|n| \geq 2$, then $M_k(n,\frac{k}{k-2}) \leq u_k$, where 
$$
u_3=15, \quad  u_4=10, \quad u_5=6, \quad \quad u_k=8 \quad \text{for} \; \;  6\leq k\leq 14, \quad \text{ and } \quad u_k=5 \quad \text{for} \; \;  k \geq 15.
$$
\end{thm}

When $|n|$ is small compared to $k$, the state of art of upper bounds on $M_k(n)$ is due to B\'{e}rczes, Dujella, Hajdu, and Luca \cite[Theorem 3]{BDHL11}, where they showed that $M_3(n) \leq 2|n|^{17}+6$ and $M_k(n) \leq 2|n|^5 + 3$ for $k\geq 5$ and $k$ is odd. Our next result substantially improves this upper bound. 
\begin{thm}\label{thm:M_k(n)explicit}
Let $k \geq 3$ and let $\frac{k}{2k-4}<L\leq \frac{k}{k-2}$ be a real number. Then for all nonzero integers $n$, we have 
$M_k(n,L) \leq T$ and 
$$
M_k(n)\leq (|n|^{2L}+n)^{1/k}+T+1,
$$
where
\begin{equation}\label{eq:T}
T=T(k,L)=\frac{3(\log 18-\log L)}{\log (k-1)-\log(3+\frac{k}{L}-k)}+12.    
\end{equation}
\end{thm}

The next corollary follows from Theorem~\ref{thm:M_k(n,3)} and Theorem~\ref{thm:M_k(n)explicit} immediately. In particular, it shows that if $k\gg \log |n|$, then $M_k(n)$ is uniformly bounded. A similar result also holds for $BM_k(n)$ in view of Theorem~\ref{thm:explicit}.

\begin{cor}
If $k \geq 3$ and $n$ is a nonzero integer, then $M_k(n) \leq 2^{1/k}|n|^{2/(k-2)}+16$. In particular, if $n,k$ are integers such that  $|n|\geq 2$ and $k \geq 2\log |n|+2$, then $M_k(n)\leq 19$.
\end{cor}

When $|n|$ is much larger compared to $k$, it is known that the upper bound on $M_k(n)$ can be further improved. The best-known upper bound on $M_2(n)$ is $(2+o(1))\log |n|$ as $|n| \to \infty$, due to Yip \cite{Y24+}. When $k \geq 3$, the best-known upper bound is due to Dixit, Kim, and Murty \cite{DKM22}, and Bhattacharjee, Dixit, and Saikia \cite{BDS23}, where they showed that if $k \geq 3$ is fixed and $|n| \to \infty$, then $M_k(n) \leq (3+o(1)) \phi(k)\log |n|$. Under the Paley graph conjecture, they showed the stronger bound $M_{k}(n)=(\log |n|)^{o(1)}$. Given inequality~\eqref{eq:compare}, Theorem~\ref{thm:BM_k} immediately recovers and improves these two results. Very recently, Kim, Yip, and Yoo \cite[equation (1.5)]{KYY} showed an improved upper bound that $M_k(n) \leq (\frac{2\phi(k)}{k-2}+o(1)) \log n$ when $n \to \infty$; in fact, their full result \cite[Theorem 1.2]{KYY} is even stronger when $k$ is composite. By adapting the proof of Theorem~\ref{thm:BM_k} to this nicer setting, we further improve their asymptotic upper bound on $M_k(n)$. 

\begin{thm}\label{thm:M_k(n)}
Let $k \geq 3$. As $|n| \to \infty$, 
 $$M_k(n) \leq \bigg(\frac{\phi(k)}{k-2}+o(1)\bigg) \log |n|.$$   
\end{thm}

\section{Preliminaries}\label{sec:prelim}
In this section, we introduce three tools required for our proof.

\subsection{Large primitive solutions for Thue inequalities}
We need the following result of Evertse \cite[Theorem 2.1]{E83} concerning explicit upper bounds on Thue inequalities.

\begin{lem}\label{Thue}
If $a, b$ and $k$ are positive integers with $k \geq 3$ and $c$ is a positive real number, then there is at most one positive integral solution $(x, y)$ to the inequality
$
\left|a x^k-b y^k\right| \leq c
$
with $\operatorname{gcd}(x, y)=1$ and
$$
\max \left\{\left|a x^k\right|,\left|b y^k\right|\right\}>\beta_k c^{\alpha_k},
$$
where $\alpha_k$ and $\beta_k$ are effectively computable positive constants satisfying
$$
\alpha_3=9, \quad \alpha_k=\max \left\{\frac{3 k-2}{2(k-3)}, \frac{2(k-1)}{k-2}\right\} \quad \text { for } k \geq 4
$$
and
$$
\beta_3=1152.2, \quad \beta_4=98.53, \quad \beta_k<k^2 \quad \text { for } k \geq 5 .
$$    
\end{lem}

The above lemma has played an important role in studying $M_k(n)$ in \cite{BD03, BDHL11}. Instead, a quantitative version of Roth's theorem \cite{E10} was employed in recent papers \cite{BDS23, DKM22, KYY}.

\subsection{Product sets in shifted multiplicative subgroups}
The following lemma is from a recent paper by Kim, Yip, and Yoo \cite[Theorem 1.1]{KYY}. Its proof is based on Stepanov's method and it plays an important role in improving the upper bound on $M_k(n)$. 
\begin{lem}\label{stepanovea}
Let $p$ be a prime and let $k \mid (p-1)$ with $k \ge 2$. Let $S_k=\{x^k: x \in \F_p^*\}$. Let $A,B \subset \F_p^*$ and $\lambda \in \F_p^*$ with $|A|, |B| \geq 2$. 
If $AB+\lambda \subset S_k \cup \{0\}$, then $$|A||B| \leq |S_k|+|B \cap (-\lambda A^{-1})|+|A|-1.$$ 
\end{lem}

Assuming the Paley graph conjecture, Lemma~\ref{stepanovea} can be significantly improved if $p$ is much larger compared to $k$. Recall that the conjecture states the following (see for example \cite{DKM22, KYY}).

\begin{conj}[Paley graph conjecture]\label{Paley-graph-conjecture}
Let $\epsilon >0$ be a real number. Then there is $p_0=p_0(\epsilon)$ and $\delta=\delta(\epsilon)>0$ such that for any prime $p>p_0$, any $A, B \subseteq \mathbb{F}_p$ with $|A|,|B| > p^{\epsilon}$, and any non-trivial multiplicative character $\chi$  of $\F_p$, the following inequality holds:
\begin{equation*}
    \bigg| \sum_{a\in A,\,  b\in B} \chi(a+b)\bigg| \leq p^{-\delta} |A| |B|.
\end{equation*}
\end{conj}

\begin{cor}\label{cor:Paley}
Assuming the Paley graph conjecture. Let $k \geq 2$ and let $\epsilon >0$ be a real number. Let $p \equiv 1 \pmod k$ be a prime and put $S_k=\{x^k: x \in \F_p^*\}$. If $A,B \subset \F_p$ and $\lambda \in \F_p^*$ such that $AB+\lambda \subset S_k \cup \{0\}$, then $\min \{|A|,|B|\} \ll_{\epsilon} kp^{\epsilon}$, where the implicit constant depends only on $\epsilon$. 
\end{cor}
\begin{proof}
Let $p_0=p_0(\epsilon)$ and $\delta=\delta(\epsilon)$ be from the Paley graph conjecture. Assume that $p>p_0$ and $\min \{|A|,|B|\}>kp^{\epsilon}+1$. Let $\chi$ be a multiplicative character of $\F_p$ with order $k$  and let $g$ be a primitive root of $\F_p$. Let $\omega=\chi(g)$; then $\omega$ is a $k$-th primitive root of $1$. For each $0 \leq i \leq k-1$, define $B_i=\{b \in B: \chi(b)=\omega^i\}$. By the pigeonhole principle, there is $0 \leq j \leq k-1$ such that $|B_j|>p^{\epsilon}$. Note that for each $a \in A$ and each $b \in B_j$, we have $\chi(ab+\lambda) \in \{0,1\}$ and thus $\chi(a+\lambda/b)=\chi(ab+\lambda) \chi(b^{-1}) \in \{0, \omega^{-j}\}$. Let $C=\{\lambda/b: b \in B_j\}$; then $|C|=|B_j|>p^{\epsilon}$. Then for each $a \in A$ and $c \in C$, we have $\chi(a+c) \in \{0,\omega^{-j}\}$. Note that the number of pairs $(a,c) \in A \times C$ such that $\chi(a+c)=0$ is at most $|A|$. It follows that
$$
   |A|(|C|-1) \leq \bigg| \sum_{a\in A,\,  c\in C} \chi(a+c)\bigg| \leq p^{-\delta} |A| |C|
$$
and thus $\frac{|C|-1}{|C|} \leq p^{-\delta}$, which is impossible if $p$ is sufficiently large. 
\end{proof}

For the applications of our techniques to Diophantine tuples with property $D_k(n)$, we also require the following lemma from \cite[Theorem 1.5 (2)]{KYY}.

\begin{lem}\label{stepanoveb}
Let $p$ be a prime and let $k \mid (p-1)$ with $k \ge 2$. Let $S_k=\{x^k: x \in \F_p^*\}$. Let $A \subset \F_p^*$ and $\lambda \in \F_p^*$. If $ab+\lambda \in S_k \cup \{0\}$ for each $a,b \in A$ with $a \neq b$, then $|A|\leq \sqrt{2(p-1)/k}+4$.     
\end{lem}

\subsection{A gap principle}
We will apply the following simple gap principle repeatedly throughout the proof. It is inspired by the work of Gyarmati \cite{G01}.

\begin{lem}\label{gap_principle}
Let $k \geq 3$ and let $n$ be a nonzero integer. Let $a,b,c,d$ are positive integers such that $a<b$, $c<d$, and $ac\geq 2|n|$. Suppose further that $ac+n, bc+n, ad+n, bd+n$ are $k$-th powers. Then $bd \geq k^{k} (ac)^{k-1}/(4^{k-1}|n|^k)$.     
\end{lem}
\begin{proof}
We first consider the case $n>0$. We note that $(ac+n)(bd+n)$ and $(ad+n)(bc+n)$ are both $k$-th powers such that $(ac+n)(bd+n) > (ad+n)(bc+n)$. Thus, 
\begin{align*}
(ac+n)(bd+n) 
&\geq  [ ((ad+n)(bc+n))^{1/k} + 1]^k\\
&\geq (ad+n)(bc+n) + k ((ad+n)(bc+n))^{(k-1)/k} \\
&\geq (ad+n)(bc+n) + k (abcd)^{(k-1)/k}.
\end{align*}
It follows that $ n(ac+bd)\geq n(ad+bc) + k(abcd)^{(k-1)/k} $ and thus $nbd \geq k (abcd)^{(k-1)/k} $, which implies that $bd \geq k^{k} (ac)^{k-1}/n^k$, as required.

Next, we consider the case $n<0$. In this case we have $(ac+n)(bd+n)<(ad+n)(bc+n)$. Note that we have $ac+n \geq ac/2$ and $bd+n \geq bd/2$. Thus, similarly as before, we have
\begin{align*}
(ad+n)(bc+n) 
&\geq  [ ((ac+n)(bd+n))^{1/k} + 1]^k\\
&\geq (ac+n)(bd+n) + k ((ac+n)(bd+n))^{(k-1)/k} \\
&\geq (ad+n)(bc+n) + k (abcd/4)^{(k-1)/k}.
\end{align*}
It follows that $|n|bd \geq k (abcd/4)^{(k-1)/k}$, which implies that $bd \geq k^{k} (ac)^{k-1}/(4^{k-1}|n|^k)$, as required.
\end{proof}

As a corollary of the above gap principle, under some additional assumptions, we show that the elements in $B$ grow super-exponentially if $AB+n$ is contained in the set of $k$-th powers.

\begin{cor}\label{cor:exp}
Let $k \geq 3$ and let $n$ be a nonzero integer. Let $L$ be a real number such that $L>\frac{k}{k-2}$. Let $A=\{a_1,a_2\}$ and $B=\{b_1, b_2, \ldots, b_m\}$ be subsets of positive integers such that $a_1<a_2$ and $b_1<b_2<\cdots<b_m$, and $ab+n$ is a $k$-th power for each $a \in A$ and $b \in B$. If $a_1^{k-1} \geq a_2$ and $b_1\geq \max\{|n|^{L},2|n|\}$, then $b_i\geq b_1^{\theta^{i-1}}$ for each $1 \leq i \leq m$, where $\theta=k-1-\frac{k}{L}>1$. 
\end{cor}
\begin{proof}
Let $1\leq i<m$. It suffices to show that $b_{i+1}\geq b_i^{\theta}$. Indeed, since $a_1^{k-1} \geq a_2$ and $b_1 \geq 2|n|$, applying Lemma~\ref{gap_principle} to $a_1,a_2, b_i, b_{i+1}$, we have
$$
b_{i+1}\geq \frac{k^{k} (a_1b_i)^{k-1}}{a_2 \cdot 4^{k-1}|n|^k} \geq \frac{b_i^{k-1}}{|n|^k}\geq \frac{b_i^{k-1}}{b_i^{k/L}}=b_i^{k-1-\frac{k}{L}}=b_i^{\theta}.
$$
In the above inequality, we used the fact that $k^k \geq 4^{k-1}$ holds for $k \geq 3$.
\end{proof}

For Diophantine tuples with property $D_k(n)$, we can deduce the following two stronger results.

\begin{cor}\label{gap_principle_Diotuple}
Let $k \geq 3$ and let $n$ be a nonzero integer. Let $A=\{a,b,c,d\}$ with property $D_k(n)$, such that $|n|^{\frac{k}{k-2}}\leq a<b<c<d$. Then $d>a^{k-1}$. 
\end{cor}
\begin{proof}
From Lemma~\ref{gap_principle} and the given assumptions, we have
$$
bd \geq k^{k} (ac)^{k-1}/(4^{k-1}|n|^k)>k^{k} (ab)^{k-1}/(4^{k-1}|n|^k)=(k^k/4^{k-1}) a^{k-1}b (b^{k-2}/|n|^k) \geq a^{k-1}b.
$$
It follows that $d>a^{k-1}$.
\end{proof}

\begin{cor}\label{gap_principle_Diotuple2}
Let $k \geq 3$ and let $n$ be a nonzero integer. Let $L$ be a real number such that $\frac{k}{2k-4}<L<\frac{k}{k-2}$. Let $A=\{a,b,c,d\}$ with property $D_k(n)$, such that $|n|^{L}\leq a<b<c<d$. Then $d>a^{\theta}$, where $\theta=(k-1)/(3+\frac{k}{L}-k)>1.$
\end{cor}
\begin{proof}
Similar to the proof of the above two corollaries, we have
$$
bd \geq k^{k} (ac)^{k-1}/(4^{k-1}|n|^k)>k^{k} (ab)^{k-1}/(4^{k-1}|n|^k)=(k^k/4^{k-1}) a^{k-1}b^{k-1}/|n|^k \geq a^{k-1}b^{k-1-k/L}.
$$
It follows that $b^{2+k/L-k}d>a^{k-1}$. Since $L<\frac{k}{k-2}$, we have $2+k/L-k>0$ and thus $d^{3+k/L-k}>b^{2+k/L-k}d>a^{k-1}$, that is, $d>a^\theta$.
\end{proof}

\section{Proofs of the main results}\label{sec:proof}
In this section, we present the proofs of our main results. We break our proofs into a few parts for clarity. Recall the following constants defined in the introduction:
\begin{align*}
&r_3=9, \quad  r_4=6, \quad r_5=5, \quad \text{ and } \quad r_k=4 \quad \text{ for } k \geq 6;\\
&s_3=6, \quad  s_4=4, \quad s_5=3, \quad \text{ and } \quad s_k=2 \quad \text{ for } k \geq 6;
\end{align*}
$$t_3=\frac{15399}{938}, \quad t_4=\frac{34}{3}, \quad t_5=\frac{97}{23}, \quad t_6=\frac{29}{4}, \text{ and } \quad t_k=\frac{k^2+k-4}{k^2-6k+6} \quad \text{ for } k \geq 7.
$$

\subsection{Bounding the contribution of large elements}
The following proposition can be viewed as a generalization and a refinement of \cite[Theorem 3]{BDHL11} and \cite[Theorem 2 and Theorem 3]{BD03}. We will use the proposition to deduce a few more refined results in Section~\ref{subsec:applications}.

\begin{prop}\label{prop:a1a2}
Let $k \geq 3$ and let $n$ be a nonzero integer. Let $A,B \subset \N$ such that $A=\{a_1, a_2, \ldots, a_{\ell}\}$ and $B=\{b_1, b_2, \ldots, b_m\}$ with $a_1<a_2<\cdots<a_{\ell}$ and $b_1<b_2<\cdots<b_m$, and $AB+n \subset \{x^k: x \in \N\}$. If $k>3$, further assume that $m \geq s_k+1$, $\ell \geq 2$, and $a_2 \leq b_{m-s_k}$; if $k=3$, further assume that $m \geq 7$, $\ell \geq 3$, and $a_3 \leq b_{m-6}$. Then at most $s_k$ elements in $B$ are at least $2|n|^{t_k}$.    
\end{prop}
\begin{proof}
For each $1 \leq i \leq m$, we can find positive integers $x_i, y_i$ such that
$$
a_1 b_i+n=x_i^k, \quad a_2 b_i+n=y_i^k,
$$
Let $d_i=\gcd(x_i, y_i)$ and write $x_i=d_i x_i'$ and $y_i=d_iy_i'$. Observe that we have
\begin{equation}\label{Thueeq}
a_2 x_i^k-a_1 y_i^k=a_2(a_1b_i+n)-a_1(a_2b_i+n)=n\left(a_2-a_1\right).    
\end{equation}
Note that $d_i^k \leq|n|\left(a_2-a_1\right)$, and the pair $(x_i', y_i')$ uniquely determines the pair $(x_i,y_i)$ in view of equation~\eqref{Thueeq}. Thus, the map $i \mapsto (x_i', y_i')$ is injective. Recall that Lemma~\ref{Thue} implies that the following Thue inequality
$$
\left|a_2 x^k-a_1 y^k\right| \leq|n|\left(a_2-a_1\right)
$$
has at most one primitive solution $(x,y)$ such that 
$$
\max \left\{\left|a_2 x^k\right|,\left|a_1 y^k\right|\right\}>\beta_k (|n|\left(a_2-a_1\right))^{\alpha_k}.
$$
Thus, there is at most one $i_0 \in \{1,2,\ldots, m\}$ such that 
$$
a_2x_{i_0}^{\prime k}>\beta_k (|n|(a_2-a_1))^{\alpha_k},
$$
where $\alpha_k$ and $\beta_k$ are defined in Lemma~\ref{Thue}.
If such $i_0$ exists and $i_0=m-1$, we swap $b_{m-1}$ and $b_m$ so that
\begin{equation}\label{m-1}
a_2x_{m-1}^{\prime k}\leq \beta_k (|n|(a_2-a_1))^{\alpha_k}< \beta_k (|n|a_2)^{\alpha_k}.    
\end{equation}

For the sake of contradiction, suppose otherwise that at least $s_k+1$ many $b_i$ 's are at least $2|n|^{t_k}$. Then $b_{m-1} \geq b_{m-s_k}\geq 2|n|^{t_k}\geq 2|n|$. It follows from inequality~\eqref{m-1} that 
$$
a_1 b_{m-1}\leq 2x_{m-1}^k=2x_{m-1}^{\prime k} d_{m-1}^k \leq 2 |n| a_2 x_{m-1}^{\prime k}< 2 |n| \cdot \beta_k \cdot\left(|n| a_2\right)^{\alpha_k},
$$
and thus the assumption that $a_2 \leq b_{m-s_k}$ implies that
\begin{equation}\label{ub}
b_{m-1}<2 \beta_k|n|^{\alpha_k+1} a_2^{\alpha_k}\leq  2 \beta_k|n|^{\alpha_k+1} b_{m-s_k}^{\alpha_k}.    
\end{equation}

To derive a contradiction, next, we provide a lower bound on $b_{m-1}$ by applying the gap principle (Lemma~\ref{gap_principle}). We divide the discussions into four cases according to the size of $k$. 

(1) The case $k \geq 6$. In this case, $s_k=2$ and Lemma~\ref{Thue} states that $\beta_k<k^2$, $\alpha_6=8/3$, and $\alpha_k=2(k-1)/(k-2)$ for $k \geq 7$. Thus inequality~\eqref{ub} becomes
\begin{equation}\label{ub1}
b_{m-1}<2 k^2|n|^{\alpha_k+1} b_{m-2}^{\alpha_k}.    
\end{equation}
On the other hand, applying Lemma~\ref{gap_principle} to $a_1,a_2, b_{m-2}, b_{m-1}$ and using the assumption $a_2 \leq b_{m-2}$, we have
\begin{equation}\label{lb1}
b_{m-1}>\frac{k^k (a_1b_{m-2})^{k-1}}{a_2\cdot 4^{k-1}|n|^k}\geq \frac{k^k b_{m-2}^{k-2}}{4^{k-1}|n|^k}> \frac{k^2 b_{m-2}^{k-2}}{|n|^k},
\end{equation}
where we used the fact that $k^{k}>k^2 \cdot 4^{k-1}$ when $k \geq 6$. 

Comparing inequality~\eqref{ub1} with inequality~\eqref{lb1}, we get
$$
2 k^2|n|^{\alpha_k+1} b_{m-2}^{\alpha_k}>b_{m-1}> \frac{k^2 b_{m-2}^{k-2}}{|n|^k}.   
$$
Thus we obtain that
$$
b_{m-2}^{k-2-\alpha_k}<2 |n|^{1+k+\alpha_k}.
$$
When $k \geq 7$, we have
$$
\frac{1+k+\alpha_k}{k-2-\alpha_k}=\frac{1+k+2(k-1)/(k-2)}{k-2-2(k-1)/(k-2)}=\frac{(1+k)(k-2)+2(k-1)}{(k-2)^2-2(k-1)}=\frac{k^2+k-4}{k^2-6k+6}=t_k.
$$
Therefore, when $k=6$, $b_{m-2}<2|n|^{29/4}=2|n|^{t_6}$; when $k \geq 7$, $b_{m-2}<2|n|^{t_k}$. Both cases violate the assumption that $b_{m-2}\geq 2|n|^{t_k}$. Hence, at most two $b_i$ 's are at least $2|n|^{t_k}$, as required.

(2) The case $k=5$. In this case, $s_5=3$ and Lemma~\ref{Thue} states that $\beta_5<25$ and $\alpha_5=13/4$, thus inequality~\eqref{ub} becomes
\begin{equation}\label{ub2}
b_{m-1}<50|n|^{17 / 4} b_{m-3}^{13 / 4}.    
\end{equation}
On the other hand, applying Lemma~\ref{gap_principle} to $a_1,a_2, b_{m-2}, b_{m-1}$ and using the assumption $a_2 \leq b_{m-3}$, we have
$$
b_{m-1}>\frac{5^5 (a_1b_{m-2})^{4}}{a_2\cdot 4^{4} \cdot |n|^5}\geq \frac{12 b_{m-2}^{3}}{|n|^5}.
$$
By the same arguments, we get $b_{m-2}>12 b_{m-3}^{3}/|n|^5$. Therefore,
\begin{equation}\label{lb2}
b_{m-1}>\frac{12 b_{m-2}^{3}}{|n|^5} >\frac{12}{|n|^5} \cdot \frac{(12 b_{m-3}^{3})^3}{(|n|^5)^3}>\frac{50 b_{m-3}^{9}}{|n|^{20}}.
\end{equation}
Comparing inequality~\eqref{ub2} with inequality~\eqref{lb2}, we get
$$
50|n|^{17 / 4} b_{m-3}^{13 / 4}>b_{m-1}> \frac{50 b_{m-3}^{9}}{|n|^{20}},   
$$
and thus $b_{m-3}^{23/ 4}<|n|^{97 / 4}$, that is, $b_{m-3}<|n|^{97/23}$, violating the assumption.  Therefore, at most three $b_i$'s are at least $2|n|^{97/23}$, as required.

(3) The case $k=4$. In this case, $s_4=4$ and Lemma~\ref{Thue} states that $\beta_4< 99$ and $\alpha_4=5$, thus inequality~\eqref{ub} becomes
\begin{equation}\label{ub3}
b_{m-1}<198|n|^{6} b_{m-4}^{5}.    
\end{equation}
Similarly to the case (2), using the gap principle and the assumption that $a_2 \leq b_{m-4}$, we have
$$
b_{m-j}>4b_{m-j-1}^2 |n|^{-4}
$$
for $j=1,2,3$. It follows that
\begin{equation}\label{lb3}
b_{m-1}>\frac{4b_{m-2}^2}{|n|^{4}}>\frac{64 b_{m-3}^4}{|n|^{12}}>\frac{16384 b_{m-4}^8}{|n|^{28}}
\end{equation}
Comparing inequality~\eqref{ub3} with inequality~\eqref{lb3}, we get $b_{m-4}<|n|^{34/3}$, violating the assumption.  Hence, at most four $b_i$ 's are at least $2|n|^{34/3}$, as required.

(4) The case $k=3$. In this case, $s_3=6$ and we assumed that $a_3 \leq b_{m-6}$. Lemma~\ref{Thue} states that $\beta_4< 1153$ and $\alpha_3=9$, thus inequality~\eqref{ub} becomes
\begin{equation}\label{ub4}
b_{m-1}<2306|n|^{10} b_{m-6}^{9}.    
\end{equation}
Applying Lemma~\ref{gap_principle} to $a_1,a_2, b_{m-2}, b_{m-1}$ and to $a_2,a_3, b_{m-2}, b_{m-1}$, we obtain
$$
b_{m-1}>\frac{27 (a_1b_{m-2})^2}{16a_2|n|^3}, \quad \text{ and } \quad b_{m-1}>\frac{27 (a_2b_{m-2})^2}{16a_3|n|^3}.
$$
If $b_{m-1}<27b_{m-2}^{5/3}/(16|n|^{3})$, then the above inequalities imply that
$$
a_2>a_1^2b_{m-2}^{1/3}, \quad \text{ and } \quad a_3>a_2^2b_{m-2}^{1/3}, 
$$
which further implies that 
$$
a_3>a_2^2b_{m-2}^{1/3}>a_1^4b_{m-2}>b_{m-2},
$$
violating the assumption that $a_3 \leq b_{m-6}$. Thus, $b_{m-1}\geq 27b_{m-2}^{5/3}/(16|n|^{3})$.

By a similar argument, for each $1\leq j \leq 5$, we have
$$b_{m-j} \geq \frac{27b_{m-j-1}^{5/3}}{16|n|^{3}},$$
or equivalently,
$$Db_{m-j} \geq (Db_{m-j-1})^{5/3},$$
where $D=(27/16)^{3/2}|n|^{-9/2}$. It follows that
\begin{equation}\label{lb4}
b_{m-1} \geq (Db_{m-6})^{(5/3)^5}/D=D^{(5/3)^5-1}b_{m-6}^{(5/3)^5}>10000 (|n|^{-9/2})^{(5/3)^5-1}b_{m-6}^{(5/3)^5}.
\end{equation}
Comparing inequality~\eqref{ub4} with inequality~\eqref{lb4}, we get $b_{m-4}<|n|^{\gamma}$, where
$$
\gamma= \frac{10+\frac{9}{2} \cdot \left(\left(\frac{5}{3}\right)^{5}-1\right)}{\left(\frac{5}{3}\right)^{5}-9}=\frac{15399}{938}=t_3,
$$
violating the assumption.  Hence, at most six $b_i$ 's are at least $2|n|^{t_3}$, as required.
\end{proof}

\subsection{Consequences of Proposition~\ref{prop:a1a2}}\label{subsec:applications}
In this section, we provide a few applications of Proposition~\ref{prop:a1a2}.

First, we consider its application to Diophantine tuples and present the proof of Theorem~\ref{thm:M_k(n,3)} and Theorem~\ref{thm:M_k(n)explicit}.

\begin{proof}[Proof of Theorem~\ref{thm:M_k(n,3)}]
Let $C$ be a subset of integers with property $D_k(n)$. First we claim that either $|C|\leq r_k$, or at most $s_k$ elements in $C$ are at least $2|n|^{t_k}$. In particular, if $|n|=1$, then the claim already implies that $|C|\leq \max \{r_k, s_k+1\}=r_k$. 

To prove the claim, first assume that $k>3$, and let $A=\{a_1, a_2\}$ and $B=\{b_1, b_2, \ldots, b_m\}$ such that $A \sqcup B=C$ and $a_1<a_2<b_1<b_2<\cdots<b_m$. If $m \leq s_k$, then trivially $|C|\leq s_k+2=r_k$. If $m \geq s_k+1$, then we have $a_2<b_1 \leq b_{m-s_k}$, and Proposition~\ref{prop:a1a2} implies that at most $s_k$ elements in $B$ are at least $2|n|^{t_k}$, and thus at most $s_k$ elements in $C$ are at least $2|n|^{t_k}$ since $s_k<m$. 

In the case $k=3$, we instead set $A=\{a_1, a_2,a_3\}$ and $B=\{b_1, b_2, \ldots, b_m\}$ such that $A \sqcup B=C$ and $a_1<a_2<a_3<b_1<b_2<\cdots<b_m$. A similar argument shows that either $|C|\leq 9=r_3$, or at most $s_3=6$ elements in $C$ are at least $2|n|^{t_k}$. 

Next assume that $|n|\geq 2$ and $C \subset [n^{\frac{k}{k-2}},\infty)$ with $|C|>r_k$. Consider the subset $C'=C \cap [1,2|n|^{t_k})$. The above claim shows that $|C|-|C'| \leq s_k$.  Let $C'=\{c_1,c_2, \ldots, c_{\ell}\}$ with $c_1<c_2<\ldots<c_\ell$. Note that $|n|^{\frac{k}{k-2}}\leq c_1$ and $c_\ell \leq 2|n|^{t_k}$. By Corollary~\ref{gap_principle_Diotuple}, we have $c_{i+3}>c_i^{k-1}$ for each $1 \leq i \leq \ell-3$. It follows that $c_{3j+1}>c_1^{(k-1)^j}$ for each $1 \leq j \leq (\ell-1)/3$. Thus
$$
|n|^{t_k+1} \geq 2|n|^{t_k} \geq c_\ell > c_1^{(k-1)^{\lfloor (\ell-1)/3\rfloor}} \geq |n|^{\frac{k}{k-2} \cdot (k-1)^{\lfloor (\ell-1)/3\rfloor}},
$$
that is, $t_k+1>\frac{k}{k-2} \cdot (k-1)^{\lfloor (\ell-1)/3\rfloor}$. Thus, by plugging in the value of $t_k$, we have: if $k \geq 15$, then $\ell \leq 3$; if $7 \leq k \geq 14$, then $\ell \leq 6$; if $k=6$, then $\ell \leq 6$; if $k=5$, then $\ell \leq 3$; if $k=4$, then $\ell \leq 6$; if $k=3$, then $\ell \leq 9$. It then follows that $|C|\leq \ell+s_k \leq u_k$, as required.
\end{proof}

\begin{rem}\label{rem:special}
In the special case $n=1$, Bugeaud and Dujella \cite[Corollary 4]{BD03} has shown that $M_3(1) \leq 7$, $M_k(1) \leq 5$ for $k \in \{4,5\}$, $M_k(1) \leq 4$ for $6 \leq k \leq 176$, and $M_k(1) \leq 3$ for $k \geq 177$, which is better than our result. \footnote{As pointed out by \cite{BDHL11}, there was a minor inaccuracy in the original proof of \cite[Corollary 4]{BD03}, but it only affected the upper bound on $M_5(1)$.}
\end{rem}

\begin{proof}[Proof of Theorem~\ref{thm:M_k(n)explicit}]
Let $L \in (\frac{k}{2k-4},\frac{k}{k-2}]$ and let $\theta=(k-1)/(3+\frac{k}{L}-k)>1$. Let $n$ be a nonzero integer. If $|n|=1$, then we can apply Theorem~\ref{thm:M_k(n,3)}. Next, we assume that $|n|\geq 2$. 

First, we give an upper bound on $M_k(n, L)$. Let $C$ be a subset of integers with property $D_k(n)$ with $|C|>9$. Let $C'=C \cap [|n|^L, |n|^{18}]$. Then $[|n|^L, 2|n|^{t_k}] \subset C'$, and we have $|C|-|C'|\leq 9$ in view of the proof of Theorem~\ref{thm:M_k(n,3)}. Let $C'=\{c_1,c_2, \ldots, c_{\ell}\}$ with $c_1<c_2<\ldots<c_\ell$. By Corollary~\ref{gap_principle_Diotuple2}, we have $c_{i+3}>c_i^{\theta}$ for each $1 \leq i \leq \ell-3$, where $\theta=(k-1)/(3+\frac{k}{L}-k)>1.$ Similar to the proof of Theorem~\ref{thm:M_k(n,3)}, we have 
$$
|n|^{18} \geq c_\ell > c_1^{\theta^{\lfloor (\ell-1)/3\rfloor}} \geq |n|^{L \cdot \theta^{\ell/3-1}},
$$
that is, $\frac{18}{L}>\theta^{\ell/3-1}$. Thus, we have $M_k(n,L)\leq T$, where
\begin{equation*}
T=T(k,L)=\frac{3\log \frac{18}{L}}{\log \theta}+12=\frac{3(\log 18-\log L)}{\log (k-1)-\log(3+\frac{k}{L}-k)}+12,    
\end{equation*}
as required. 

Next we bound $M_k(n)$. Let $C$ be a maximum subset of integers in $[1, |n|^L]$ with property $D_k(n)$; then we have $M_k(n)\leq |C|+M_k(n,L)\leq |C|+T$. Let $c \in C$; then for each $c' \in C$ such that $c' \neq c$, $cc'+n$ is a $k$-th power with $cc'+n\leq |n|^{2L}+n$. It follows that $|C|-1$ is upper bounded by the number of $k$-th powers not exceeding $|n|^{2L}+n$. We conclude that $$M_k(n)\leq |C|+T \leq (|n|^{2L}+n)^{1/k}+T+1$$
as required.
\end{proof}

Next, we consider the applications to bipartite Diophantine tuples.

\begin{prop}\label{prop:large}
Let $k \geq 3$ and let $n$ be a nonzero integer. Let $A,B \subset \N$ such that $AB+n \subset \{x^k: x \in \N\}$. If $\min \{|A|, |B|\} \geq r_k$, then in both sets $A$ and $B$, at most $s_k$ elements are at least $2|n|^{t_k}$.
\end{prop}

\begin{proof}
Let $A=\{a_1, a_2, \ldots, a_{\ell}\}$ and $B=\{b_1, b_2, \ldots, b_m\}$ with $a_1<a_2<\cdots<a_{\ell}$ and $b_1<b_2<\cdots<b_m$. If $k \geq 4$, we may additionally assume $a_2\leq b_{m-s_k}$ without loss of generality; indeed, if $a_2>b_{m-s_k}$, then we have $b_2\leq b_{m-s_k}<a_2\leq a_{\ell-s_k}$ since $\min \{|A|, |B|\} \geq r_k=s_k+2$, and we can switch the role of $A$ and $B$. If $k=3$, we may additionally assume $a_3 \leq b_{m-6}$ using a similar argument. It follows from Proposition~\ref{prop:a1a2} that at most $s_k$ elements in $B$ are at least $2|n|^{t_k}$.

Now we have proved the required statement for the set $B$, we move to the discussion on the set $A$. We only prove the case that $k \geq 4$ (the proof of the case $k=3$ is similar). Since $|B|-s_k \geq r_k-s_k \geq 2$, at least two elements in $B$ are at most $2|n|^{t_k}$. In particular, $b_2 \leq 2|n|^{t_k}$. We consider the following two situations:
\begin{itemize}
    \item If $b_2>a_{\ell-s_k}$, then in particular we have $a_{\ell-s_k}< b_2\leq 2|n|^{t_k}$, and thus it follows immediately that at most $s_k$ elements in $A$ are at least $2|n|^{t_k}$.
    \item If $b_2 \leq a_{\ell-s_k}$, then by swapping the role of $A$ and $B$ and applying Proposition~\ref{prop:a1a2}, we deduce that at most $s_k$ elements in $A$ are at least $2|n|^{t_k}$.
\end{itemize} 
In both situations, we get the required statement.
\end{proof}

Combining the gap principle with the above proposition, we obtain the following corollary. 

\begin{cor}\label{cor:strong}
Let $k \geq 3$, $n$ be an integer such that $|n| \geq 2$, and $L$ be a real number such that $L>\frac{k}{k-2}$. Let $A,B \subset \N$ such that $AB+n \subset \{x^k: x \in \N\}$.  If $\min \{|A|, |B|\} \geq \frac{\log \log |n|+3.3}{\log (k-1)}+8$, then in both sets $A$ and $B$, at most $\frac{\log 18}{\log \theta}+7$ elements are at least $\max\{2|n|,|n|^L\}$, where $\theta=k-1-\frac{k}{L}>1$.
\end{cor}
\begin{proof}
Since $\min \{|A|, |B|\} \geq 9$, Proposition~\ref{prop:large} implies that there are at most $6$ elements in $A$ that are at least $2|n|^{17}$, and the same is true for $B$. Let $A'=A \cap [1, 2|n|^{17}]$ and $B'=B \cap [\max\{2|n|,|n|^L\},2|n|^{17}]$. Let $A'=\{a_1, a_2, \ldots, a_{\ell}\}$ and $B'=\{b_1, b_2, \ldots, b_m\}$ with $a_1<a_2<\cdots<a_{\ell}$ and $b_1<b_2<\ldots<b_m$. 

We claim that there exists $1 \leq i_0 \leq \ell-1$ with $a_{i_0}^{k-1} \geq a_{i_0+1}$. Suppose otherwise, then we have $\log a_{i+1} > (k-1)\log a_i$ for each $1 \leq i \leq \ell-1$. It follows that 
$$\log (|n|^{18}) \geq \log (2|n|^{17}) \geq \log a_{\ell} \geq (k-1)^{\ell-2} \log a_2 \geq (k-1)^{\ell-2} \log 2,$$
and thus 
$$
\log 18+\log \log |n| \geq (\ell-2) \log (k-1)+ \log \log 2.
$$
We conclude that
$$|A| \leq |A'|+6= \ell +6 < \frac{\log \log |n|+3.3}{\log (k-1)}+8,
$$
violating the assumption. 

Now choose $1 \leq i_0 \leq \ell-1$ such that $a_{i_0}^{k-1} \geq a_{i_0+1}$. Applying Corollary~\ref{cor:exp} to the set $\{a_{i_0}, a_{i_0+1}\}$ and $B'$, we have $b_j \geq b_1^{\theta^{j-1}}$ for each $1 \leq j \leq m$, where $\theta=k-1-\frac{k}{L}$. It follows that
$$18\log |n|\geq \log (2|n|^{17}) \geq \log b_m \geq \theta^{m-1} \log b_1>\theta^{m-1} \log |n|,$$
and thus
$$
m<\frac{\log 18}{\log \theta}+1.
$$
It follows that at most $\frac{\log 18}{\log \theta}+7$ elements in $B$ are at least $\max\{2|n|,|n|^L\}$. By switching the role of $A$ and $B$, we can prove the same result for $A$.
\end{proof}

Now we are ready to prove Theorem~\ref{thm:explicit}.
\begin{proof}[Proof of Theorem~\ref{thm:explicit}]
When $|n|=1$, the statement follows from Proposition~\ref{prop:large}. 

Next, we assume that $|n| \geq 2$. Let $A,B \subset \N$ such that $AB+n \subset \{x^k: x \in \N\}$. We can assume that $\min \{|A|, |B|\} \geq \frac{\log \log |n|+3.3}{\log (k-1)}+8$, for otherwise we are done. Applying Corollary~\ref{cor:strong} with $L=\frac{k+1}{k-2}$, at most $\frac{\log 18}{\log (5/4)}+7<20$ elements in $A$ are at least $\max\{2|n|,|n|^L\}$, and the same is true for $B$. Similar to the arguments used in the proof of Theorem~\ref{thm:M_k(n)explicit}, we can bound the number of elements less than $\max\{2|n|,|n|^L\}$, and conclude that $\min \{|A|, |B|\}\leq (\max\{2|n|,|n|^L\}^2+n)^{1/k}+20$, as required.
\end{proof}

\subsection{Bounding the contribution of small elements}\label{subsec:small}

It remains to estimate the contribution of small elements. To achieve that, we first recall Gallagher's larger sieve \cite{G71}.
\begin{lem}[Gallagher's larger sieve]\label{GS}  
Let $N \in \N$ and $A\subset\{1,2,\ldots, N\}$. Let ${\mathcal P}$ be a set of primes. For each prime $p \in {\mathcal P}$, let $A_p=A \pmod{p}$. For any $1<Q\leq N$, we have
$$
 |A|\leq \frac{\underset{p\leq Q, p\in \mathcal{P}}\sum\log p - \log N}{\underset{p\leq Q, p \in \mathcal{P}}\sum\frac{\log p}{|A_p|}-\log N},
$$
provided that the denominator is positive.
\end{lem}

\begin{prop}\label{prop:GS}
Let $L>\frac{1}{2}$ and $\epsilon>0$ be real numbers. Let $k \geq 3$ and let $n$ be a nonzero integer. Let $A, B$ are two nonempty subsets of $\{1,2,\ldots, |n|^L\}$ such that $AB+n$ is contained in the subset of $k$-th powers. Then 
\begin{enumerate}
    \item $\max \{|A|,|B|\} \ll |n|^{\frac{L}{k}+\epsilon}$;
    \item $\min \{|A|,|B|\} \leq (4+o(1))L\frac{\phi(k)}{k}\log |n|$,
    as $|n| \to \infty$; moreover, assuming the Paley graph conjecture, a stronger bound holds: $\min \{|A|,|B|\}=(L\log |n|)^{o(1)}.$
\end{enumerate}
The implicit constants above are absolute.
\end{prop}
\begin{proof}
Let $N=|n|^L$. 

(1) Without loss of generality, assume that $|A|\leq |B|$. Let $a \in A$. Then for each $b \in B$, $ab+n$ is a $k$-th power in the arithmetic progression $n+a, n+2a, \ldots, n+aN$ of length $N$. Thus, the result of Bourgain and Demeter \cite{BD18} implies that 
$|B|\ll d(a)^{k-1} N^{\frac{1}{k}}(\log N)^{O(1)},$
where $d(a)$ is the number of divisors of $a$. Since $a \leq N$, the bound on the divisor function $d(a)=a^{o(1)}=N^{o(1)}$ implies that $|B|\ll N^{\frac{1}{k}+\epsilon/L}=|n|^{\frac{L}{k}+\epsilon}$.

(2) We apply Gallagher's larger sieve to study $\min \{|A|, |B|\}$. Consider the set of primes 
$$\mathcal{P}=\{p \leq Q: p \equiv 1 \pmod k, p \nmid n\},$$
where $Q=\frac{4}{k}(\phi(k)\log N)^2$.

By the prime number theorem and standard partial summation, we have
\begin{equation*}
\sum_{\substack{p \equiv 1 \pmod k\\p \le Q}}\log p \sim \frac{Q}{\phi(k)}, \quad \sum_{\substack{p \equiv 1 \pmod k\\p \le Q}} \frac{\log p}{\sqrt{p}} \sim \frac{2\sqrt{Q}}{\phi(k)}.
\end{equation*}
It follows that 
\begin{equation}\label{eq:numerator}
\sum_{p \in \mathcal{P}}\log p\leq \sum_{\substack{p \equiv 1 \pmod k\\p \le Q}} \log p= (1+o(1)) \frac{Q}{\phi(k)},
\end{equation}
\begin{equation}\label{eq:partial}
\sum_{p \in \mathcal{P}} \frac{\log p}{\sqrt{p}} \geq \sum_{\substack{p \equiv 1 \pmod k\\p \le Q}} \frac{\log p}{\sqrt{p}}-\sum_{p \mid n} \frac{\log p}{\sqrt{p}} \geq (1+o(1))\frac{2\sqrt{Q}}{\phi(k)}-(\log |n|)^{1/2}=(1+o(1))\frac{2\sqrt{Q}}{\phi(k)}
\end{equation}
using the fact that the contribution of prime divisors of $|n|$ is negligible \cite[Lemma 2.8]{KYY}, that is, $\sum_{p \mid n} \log p/\sqrt{p}\ll (\log |n|)^{1/2}=o(\log N)$.

Let $p \in \mathcal{P}$. Let $A_p=A \pmod p$ and $B_p=B \pmod p$. We can view $A_p, B_p$ as subsets of $\F_p$ such that $A_p B_p+n \subset \{x^k: x \in \F_p\}$. Since $n \neq 0$ in $\F_p$, Lemma~\ref{stepanovea} implies that $\min \{|A_p|, |B_p|\}< \sqrt{p/k}+2$. 

We partition the set $\mathcal{P}$ into two subsets:
$$
\mathcal{P}_A=\{p \in \mathcal{P}: |A_p|\leq |B_p|\}, \quad \mathcal{P}_B=\{p \in \mathcal{P}: |A_p|>|B_p|\}.
$$
From inequality~\eqref{eq:partial}, we have
\begin{align*}
\sum_{p \in \mathcal{P_A}} \frac{\log p}{|A_{p}|}+\sum_{p \in \mathcal{P_B}} \frac{\log p}{|B_{p}|}
=\sum_{p \in \mathcal{P}} \frac{\log p}{\min \{|A_p|, |B_p|\}} 
\geq (1+o(1))\sum_{p \in \mathcal{P}} \frac{\log p}{\sqrt{p/k}}
=(2+o(1)) \frac{\sqrt{kQ}}{\phi(k)}.    
\end{align*}
Without loss of generality, we may assume that
$
\sum_{p \in \mathcal{P_A}} \frac{\log p}{|A_{p}|} \geq \sum_{p \in \mathcal{P_B}} \frac{\log p}{|B_{p}|},
$
so that
\begin{equation}\label{eq:denominator}
\sum_{p \in \mathcal{P_A}} \frac{\log p}{|A_{p}|} \geq (1+o(1)) \frac{\sqrt{kQ}}{\phi(k)}.
\end{equation}
Applying Gallagher's sieve (Lemma~\ref{GS}), inequality~\eqref{eq:numerator}, and inequality~\eqref{eq:denominator}, we obtain that
\begin{align*}
|A| 
&\le \frac{\sum_{p \in \mathcal{P_A}}\log p - \log N}{\sum_{p \in \mathcal{P_A}} \frac{\log p}{|A_{p}|}-\log N} \leq \frac{\sum_{p \in \mathcal{P}}\log p - \log N}{\sum_{p \in \mathcal{P_A}} \frac{\log p}{|A_{p}|}-\log N} \leq \frac{(1+o(1))\frac{Q}{\phi(k)}-\log N}{(1+o(1)) \frac{\sqrt{kQ}}{\phi(k)}-\log N}\\
&=\frac{\frac{(4+o(1))\phi(k)}{k}(\log N)^2-\log N}{(2+o(1))\log N-\log N}
=(4+o(1))\frac{\phi(k)}{k} \log N= (4+o(1))L \frac{\phi(k)}{k} \log |n|.    
\end{align*}

Finally, we briefly discuss how to get the stronger bound assuming the Paley graph conjecture. Let $\epsilon>0$. Corollary~\ref{cor:Paley} implies that there is $C_\epsilon>0$, such that $\min \{|A_p|, |B_p|\}\leq C_\epsilon p^{\epsilon}$ for each $p \in \mathcal{P}$. Again, a standard partial summation gives
$$
\sum_{\substack{p \equiv 1 \pmod k\\p \le Q}} \frac{\log p}{p^\epsilon} \sim \frac{Q^{1-\epsilon}}{(1-\epsilon)\phi(k)}.
$$
Thus, if we instead set $$Q= (4C_\epsilon(1-\epsilon)\phi(k)\log N)^{\frac{1}{1-\epsilon}},$$ 
a similar computation shows that $$\min \{|A|, |B|\} \ll_{\epsilon} (\log N)^{\frac{\epsilon}{1-\epsilon}}.$$ Letting $\epsilon \to 0$, we get the required estimate $\min \{|A|, |B|\}=(\log N)^{o(1)}=(L \log |n|)^{o(1)}$. 
\end{proof}

For Diophantine tuples with property $D_k(n)$, there is no need to partition the set $\mathcal{P}$ into two subsets, and an almost identical argument as above leads to the following proposition. The only difference is that we shall use Lemma~\ref{stepanoveb} instead of Lemma~\ref{stepanovea}.
\begin{prop}\label{prop:Dio2}
Let $L>\frac{1}{2}$. Let $k \geq 3$ and let $n$ be a nonzero integer. Let $A$ be a maximum subset of $\{1,2,\ldots, |n|^L\}$ with property $D_k(n)$. Then 
$$
        |A| \leq (2+o(1))L\frac{\phi(k)}{k}\log |n|.
$$
\end{prop}

Now we are ready to prove Theorem~\ref{thm:M_k(n)}.
\begin{proof}[Proof of Theorem~\ref{thm:M_k(n)}]
Let $L>\frac{k}{2k-4}$ be a real number. By Theorem~\ref{thm:M_k(n)explicit}, there is a constant $T=T(k,L)$, such that $M_k(n,L)\leq T$ for all nonzero $n$. It then follows from Proposition~\ref{prop:Dio2} that $$M_k(n)\leq (2+o(1))L\frac{\phi(k)}{k}\log |n|+T=(2+o(1))L\frac{\phi(k)}{k}\log |n|$$
as $|n| \to \infty$. Letting $L \to \frac{k}{2k-4}$, we get the desired bound on $M_k(n)$. 
\end{proof}

\subsection{Finishing the proof}\label{subsec:mainproof}

We conclude the paper by finishing the proof of Theorem~\ref{thm:BM_k} and Theorem~\ref{thm:PM_k} regarding the estimates on $BM_k(n)$ and $PM_k(n,R)$.

\begin{proof}[Proof of Theorem~\ref{thm:BM_k}]
In view of Theorem~\ref{thm:explicit}, we may assume that $|n|>e^k$ so that $|n|^{\frac{k}{k-2}}>2|n|$. Let $A,B \subset \N$ such that $AB +n\subset \{x^k: x \in \N\}$ and $\min \{|A|, |B|\}=BM_k(n)$. If $BM_k(n)=\min \{|A|, |B|\}<\frac{\log \log |n|+3.3}{\log (k-1)}+8$, then we are done. 

Next assume that $\min \{|A|, |B|\} \geq \frac{\log \log |n|+3.3}{\log (k-1)}+8$. Let $L$ be a real number such that $L>\frac{k}{k-2}$. By Corollary~\ref{cor:strong}, in both sets $A$ and $B$, at most $\frac{\log 18}{\log \theta}+7$ elements are at least $\max(|n|^L, 2|n|)=|n|^L$, where $\theta=k-1-\frac{k}{L}$. Let $A'=A \cap [1, |n|^L]$ and $B'=B \cap [1, |n|^L]$. Then Proposition~\ref{prop:GS} implies that
$$
\min \{|A|, |B|\} \leq \min \{|A'|, |B'|\}+\frac{\log 18}{\log \theta}+7 \leq (4+o(1))L\frac{\phi(k)}{k}\log |n|
$$
as $|n| \to \infty$. Letting $L \to \frac{k}{k-2}$, we obtain that the desired bound on $BM_{k}(n)$.
\end{proof}

\begin{rem}\label{composite}
When $k$ is a composite number, one can obtain an improved upper bound on $BM_k(n)$ by refining the proof of Proposition~\ref{prop:GS}. Note that we only considered primes in the residue class $1 \pmod k$. To optimize the upper bound on $BM_k(n)$, we can take advantage of the flexibility in choosing the set of primes when applying Gallagher's larger sieve. In particular, we can include primes in other reduced residue classes into consideration. However, this requires a much more delicate computation. Such an idea can be made precise by following the discussion in \cite[Section 5]{KYY}. 
\end{rem}

Finally, we use a similar idea to prove Theorem~\ref{thm:PM_k}.

\begin{proof}[Proof of Theorem~\ref{thm:PM_k}]
(1) Let $A,B \subset \N$ such that $AB +n\subset \{x^k: x \in \N\}$ and $\min \{|A|, |B|\} \geq r_k$, which maximizes $|A||B|$. Let $A'=A \cap [1, 2|n|^{t_k}]$ and $B'=B \cap [1,2|n|^{t_k}]$. By Proposition~\ref{prop:large}, $A'$ and $B'$ are nonempty and $|A|-|A'|\ll 1$ and $|B|-|B'| \ll 1$. By Theorem~\ref{thm:BM_k}, $\min \{|A|, |B|\} \ll \log |n|$. By Proposition~\ref{prop:GS}, $\max \{|A|, |B|\} \ll |n|^{\frac{t_k}{k}+\epsilon}$ for each $\epsilon>0$. Therefore, for each $\epsilon>0$, 
$$
PM_k(n,r_k)=|A||B|=\min \{|A|, |B|\} \max \{|A|, |B|\}\ll |n|^{\frac{t_k}{k}+\epsilon}.
$$

(2) Let $A,B \subset \N$ such that $AB +n\subset \{x^k: x \in \N\}$ and $\min \{|A|, |B|\} \geq \frac{\log \log |n|+3.3}{\log (k-1)}+8$, which maximizes $|A||B|$. 

Let $\epsilon>0$ and let $L=\frac{k}{k-2}+\epsilon$. By Corollary~\ref{cor:strong}, at most $\frac{\log 18}{\log \theta}+7$ elements in $A$ are at least $\max\{|n|^L, 2|n|\}$ and the same statement is also true for $B$, where $\theta=k-1-\frac{k}{L}$. By Theorem~\ref{thm:BM_k}, $\min \{|A|, |B|\} \ll \log |n|$. 
Since $\max\{|n|^L, 2|n|\} \ll |n|^L$, by Proposition~\ref{prop:GS}, $\max \{|A|, |B|\}\ll |n|^{\frac{L}{k}+\frac{\epsilon}{k}} \ll |n|^{\frac{1}{k-2}+\frac{2\epsilon}{k}}$. It follows that
\[
PM_k\bigg(n,\frac{\log \log |n|+3.3}{\log (k-1)}+8\bigg)=|A||B|=\min \{|A|, |B|\} \max \{|A|, |B|\}\ll |n|^{\frac{1}{k-2}+\epsilon}.
\]
\end{proof}

\begin{rem}\label{rem:ordering}
We have shown that unconditionally, $PM_{k}(n,r_k)<\infty$ for each $k \geq 3$ and $n \neq 0$. In particular, if $k \geq 6$, then $PM_{k}(n,4)<\infty$ for each $n \neq 0$. On the other hand, under the Uniformity Conjecture~\cite{CHM}, if $k \geq 5$, by considering superelliptic curves of the form $C:y^k=(a_1x+n)(a_2x+n)$, we deduce there is a constant $C_k$ such that $PM_k(n,2)<C_k$ for all nonzero integers $n$. However, it does not seem clear whether the ABC conjecture implies that $PM_k(n,2)<\infty$ if $k \geq 3$ and $n \neq 0$. The ABC conjecture is more frequently assumed in the study of Diophantine tuples and their variants, and they often lead to conditional results on the finiteness. The major barrier seems to be the case that $A$ has only two elements $a_1, a_2$ such that $a_2$ is much larger compared to $a_1$, and all elements of $B$ lie in the interval $[a_1,a_2]$, in which case the existing techniques used to prove similar conditional results (see for example \cite[Theorem 4]{BDHL11} and \cite{L05}) do not seem to extend.     
\end{rem}
    
\section*{Acknowledgements}
The author thanks Anup Dixit, Andrej Dujella, Xiang Gao, Seoyoung Kim, Greg Martin, Jozsef Solymosi, and Semin Yoo for helpful discussions. The author is also grateful to anonymous referees for their valuable comments and suggestions.

\bibliographystyle{abbrv}
\bibliography{references}

\end{document}